\documentclass[11pt]{amsart}
\usepackage[utf8]{inputenc}
\usepackage{amsfonts, amssymb, amsmath, amsthm, color, float,enumerate}
\usepackage{url}
\usepackage{bm}
\usepackage[unicode,psdextra]{hyperref}
\usepackage{pgfplots,tikz}
\pgfplotsset{compat=1.18}

\title[Mixed weak type inequalities for pairs of weights]{Mixed weak type inequalities for pairs of weights related to the Hardy-Littlewood maximal function, Calderón-Zygmund operators and their commutators.}
\author{}
\date{}

\usepackage[a4paper, left=2cm, right=2cm, top=3cm, bottom=3cm]{geometry} 

\theoremstyle{plain}
   \newtheorem{teo}{Theorem}
   
   \newtheorem{lema}[teo]{Lemma}
   \newtheorem{propo}[teo]{Proposition}
   
\theoremstyle{definition}
   
\theoremstyle{remark}
 \newtheorem{obs}{Remark}

\numberwithin{equation}{section}
\numberwithin{teo}{section}
\allowdisplaybreaks

\definecolor{aquamarine}{rgb}{0.5, 1.0, 0.83}
\definecolor{americanrose}{rgb}{1.0, 0.01, 0.24}
\definecolor{arsenic}{rgb}{0.23, 0.27, 0.29}
\definecolor{blizzardblue}{rgb}{0.67, 0.9, 0.93}
\definecolor{blush}{rgb}{0.87, 0.36, 0.51}
\definecolor{celestialblue}{rgb}{0.29, 0.59, 0.82}
\definecolor{chocolate(web)}{rgb}{0.82, 0.41, 0.12}
\definecolor{brightpink}{rgb}{1,0,0.5}
\definecolor{cadmiunred}{rgb}{0.89,0,0.13}

\hypersetup{
	colorlinks = true,
	linkcolor = blue,
	anchorcolor = blue,
	citecolor = blue,
	filecolor = blue,
	urlcolor = blue
}

\begin{document}

\author[R. Ayala]{Rocío Ayala}
	\address{Rocío Ayala, CONICET and Departamento de Matem\'{a}tica (FIQ-UNL), Santa Fe, Argentina.}
	\email{rocioayalazara@gmail.com}
	
\author[F. Berra]{Fabio Berra}
\address{Fabio Berra, CONICET and Departamento de Matem\'{a}tica (FIQ-UNL), Santa Fe, Argentina.}
\email{fberra@santafe-conicet.gov.ar}

\author[G. Pradolini]{Gladis Pradolini}
\address{Gladis Pradolini, CONICET and Departamento de Matem\'{a}tica (FIQ-UNL), Santa Fe, Argentina.}
\email{gladis.pradolini@gmail.com}

\thanks{The authors were supported by Consejo Nacional de Investigaciones Científicas y Técnicas (CONICET), Universidad Nacional del Litoral and Gobierno de la Provincia de Santa Fe (Argentina)}



\begin{abstract}	
We study two-weight weak-type estimates for the operator $S_v f = \mathcal{T}(fv)/v$, where $\mathcal{T}$ is the Hardy-Littlewood maximal operator or a Calderón-Zygmund operator (CZO) and $v$ is a weight. Concretely, under certain conditions on the weights involved, we prove that $S_v$ is bounded from $L^{1}(wv)$ to $L^{1,\infty} (uv)$. We also consider the corresponding inequalities when $\mathcal{T}$ is a higher-order commutator of a CZO. These types of results are inspired by the article of Sawyer in  \cite{S85}, (see also \cite{MW77}).
\end{abstract}
\maketitle

\section{Introduction and main results}
An interesting result proved by E. Sawyer in \cite{S85} states that the inequality
\begin{align}\label{aux 1}
    uv\left(\left\{x\in\mathbb{R}: \frac{M(fv)(x)}{v(x)}>\lambda\right\}\right)\leq \frac{C}{\lambda}\int_{\mathbb{R}}|f(x)|u(x)v(x)\,dx
\end{align}
holds for any $A_1$ weights $u$ and $v$ and every $\lambda>0$, where $M$ is the classical Hardy-Littlewood maximal function. Although the inequality above is similar to a weak $(1,1)$ type estimate for the operator $S_v f=M(fv)/v$ with respect to the measure $d\mu(x)=u(x)v(x)\,dx$, it involves a more subtle meaning. In fact, the operator $S_v$ is a perturbation of $M$ by $v$, and this weight is also related to the measures involved. Therefore the techniques required for the treatment of this estimate turn out to be more sophisticated. Sawyer, inspired  by the techniques of principal cubes in \cite{MW77}, provided a more intricate argument than the classical approach known for weak-type estimates. The main motivation in \cite{S85} for considering this type of estimate is that it leads to an alternative proof of the continuity of $M$ in $L^p(w)$ when $w$ is an $A_p$ weight. Let us observe that, when $v=1$, \eqref{aux 1} shows that $M$ is of weak $(1,1)$ type for $u \in A_1$, which had been previously proved in \cite{MuckenhouptOG}.

Later on, Cruz-Uribe, Martell and Pérez (\cite{CUMPconjecture}) extended \eqref{aux 1} to higher dimensions, while also including Calderón-Zygmund operators (CZOs). The authors prove the estimates not only when $u$ and $v$ are $A_1$ weights, but also when $u\in A_1$ and $v\in A_\infty(u)$. 

An extension to higher-order commutators of CZOs with a BMO symbol was proved in \cite{BCP19}. The result includes the operator itself, which allows the authors to obtain the same estimate as in \cite{CUMPconjecture} by avoiding the usage of an extrapolation result. Related results concerning these types of estimates can be found in  \cite{ABP25}, \cite{B19},  \cite{BCP19a}, \cite{BCP22}, \cite{KOP}, \cite{LOP19} and \cite{ORR25}.

In this article we are interested in two-weight mixed type inequalities for the operator $\mathcal{T}$, when $\mathcal{T}$ is $M$ or a CZO. We prove that, under certain hypotheses on $(u,w)$ and $v$, the operator $S_v f = \mathcal{T}(fv) /v$ is bounded from $L^{1}(wv)$ to $L^{1,\infty}(uv)$. Concretely, we give conditions on $(u,w)$ and $v$ such that the inequality
\[uv\left(\left\{x\in\mathbb{R}: \frac{|\mathcal{T}(fv)(x)|}{v(x)}>\lambda\right\}\right)\leq \frac{C}{\lambda}\int_{\mathbb{R}^n}|f(x)|w(x)v(x)\,dx\]
holds for every $\lambda>0$.

In the case of higher-order commutators, we establish conditions on $(u,w)$ and $v$ such that the estimate
\[uv\left(\left\{x\in\mathbb{R}: \frac{|T_b^m(fv)(x)|}{v(x)}>\lambda\right\}\right)\leq C\int_{\mathbb{R}^n}\Phi_m\left(\frac{|f(x)|}{\lambda}\right)w(x)v(x)\,dx,\]
holds for every positive $\lambda$, where $\Phi_m(t)=t(1+\log^+t)^{m}$. This is an expected result since it is well-known that $T_b^m$ does not satisfy the weak $(1,1) $ type inequality.


In order to state our main results we shall give some definitions that will be useful. 

Let $T$ be a linear operator, bounded on $L^2$ that verifies
\begin{equation}\label{CZO definicion integral}
    Tf(x) =\int_{\mathbb{R}^{n}} K(x-y) f(y) \,dy 
\end{equation}
for $f \in C_{0}^{\infty}$ and $x \notin \text{supp}f$, where 
$K:\mathbb{R}^{n}\backslash \{\textbf{0}\} \to \mathbb{C}$ is a standard kernel, that is a measurable function that satisfies the following size condition
\begin{equation*}\label{condicion de tamanio}
    |K(x)| \leq \frac{C}{|x|^{n}} \hspace{7mm}\text{ for }x \in \mathbb{R}^n    \backslash\{\textbf{0}\}
\end{equation*} 
and the smoothness condition
\begin{equation}\label{condicion de suavidad}
    |K(x-y) - K(x-z)| \leq C \frac{|x-z|}{|x-y|^{n+1}} \hspace{5mm} \text{if } |x-y| > 2 |y-z|.
\end{equation}The operator $T$ is called a Calderón-Zygmund operator (CZO). Given $b \in L^{1}_{\text{loc}}$, the first order commutator of $T$ is formally defined by 
\begin{equation*}
    T_{b} f = [b,T]f= bTf - T(bf).
\end{equation*}
The higher order commutators of $T$ are recursively defined by
\begin{equation*}
    T^{m}_{b}f=[b, T^{m-1}_{b}]f, \hspace{5mm} \hspace{5mm} m \in \mathbb{N}.
\end{equation*} We also understand $T^{0}_{b}= T$.

We say that a locally integrable function $b$ belongs to the bounded mean oscillation space, BMO, if the quantity
\begin{equation*}
    \|b \|_{\text{BMO}}= \sup_{Q}\frac{1}{|Q|} \int_{Q} |b(x) -b_{Q}| \,dx
\end{equation*} is finite, where $\displaystyle b_{Q} = \frac{1}{|Q|} \int_{Q} b(x) \,dx$. 
\vspace{5mm}

The first main result establishes a mixed weak-type inequality with pairs of weights for the Hardy-Littlewood maximal operator. In order to state it, we introduce the following definition. Given a weight $z$, we say that $(u,w)$ belongs to the $A_{1}(z)$-class if there exists a positive constant $C$ that verifies
\begin{equation*}
    \frac{1}{z(Q)} \int_{Q}u(x)z(x)\,dx \leq C \inf_{Q}w
\end{equation*} for every cube $Q \subset \mathbb{R}^{n}$ with sides parallel to the coordinate axes. For more details see Section \ref{section preliminares}.
\begin{teo}\label{teoremamaximal}
   Let $1< q < \infty$, $v \in RH_{\infty} \cap A_{q'}$ and $(u,w) \in A_{1}(v^{1-q})$. Then there exists a positive constant $C$ such that the inequality
    \begin{align*}
        uv\left( \left\{ x \in \mathbb{R}^n : \frac{M(fv)(x)}{v(x)} > \lambda \right\} \right) \leq \frac{C}{\lambda} \int_{\mathbb{R}^n} |f(x)| w(x) v(x) \,dx
    \end{align*} holds for every positive $\lambda$.
\end{teo}

Let us observe that if we take $u=w$, the theorem above provides us with the mixed weak-type inequality obtained in \cite{CUMPconjecture}, under different conditions on the weights. Indeed, one of the conditions considered in \cite{CUMPconjecture} is given by $u \in A_1$ and $v \in A_{\infty}(u)$. Particularly, the latter implies that $v \in A_{q'}(u)$ for some $q>1$ and, consequently, $v^{1-q} \in A_{q}(u) \subset A_{\infty}(u)$. By applying Proposition \ref{pesos prop u A1 v en A infty u} we get that $u \in A_{1}(v^{1-q}),$ which is the condition given in Theorem \ref{teoremamaximal} for $u=w$, making our hypotheses to be more general in this sense, regardless of the properties on $v$.

On the other hand, if we take $v=1$, Theorem \ref{teoremamaximal} gives the weak-type inequality with pair of weights for $M$ proved in  \cite{MuckenhouptOG}.

The following estimate states a similar result as in the theorem above for Calderón-Zygmund operators and their commutators with BMO symbols. However, the pair of weights involved belongs to a slight perturbation of the $A_{1}$-class in the scale of Orlicz spaces related to a certain Young function $\Psi$ and a weight $z$. This class is called $A_{\Psi}(z)$ and collects the pairs of weights $(u,w)$ such that
\begin{align*}
   \sup_{Q} \|u \|_{\Psi,Q,z} \| w^{-1} \chi_{Q}\|_{\infty} < \infty.
\end{align*}Particularly, we shall consider $\Psi(t) = \Phi_{\alpha}(t)=t \log(1+t)^{\alpha}$, $\alpha \geq 0$  (see Section \ref{section preliminares} for details).
\begin{teo}\label{mixta conmutador (1,1)}
    Let $\varepsilon>0$, $m \in \mathbb{N}_{0}$, $q>1$, $v\in RH_{\infty} \cap A_{q'}$ and $(u,w) \in A_{\Phi_{m+\varepsilon}}(v^{1-q})$. If $T$ is a Calderón-Zygmund operator and $b \in $ BMO, then the inequality 
    \begin{align*}
        uv\left( \left\{ x \in \mathbb{R}^n : \frac{|T_{b}^{m}(fv)(x)|}{v(x)} > \lambda \right\} \right) \leq C \int_{\mathbb{R}^n}  \Phi_{m} \left( \frac{\|b\|_{\text{BMO}}^{m} |f(x)|}{\lambda} \right) w(x) v(x) \,dx
    \end{align*}
  holds for every positive $\lambda$.
\end{teo}
If $m=0$ and $v=1$, the inequality above provides us with a two-weight weak $(1,1)$ type estimate with pairs of weights for CZOs. Moreover, the condition $(u,w) \in A_{\Phi_\varepsilon}$ can be seen as the limiting case of the hypotheses on the pair of weights required for the corresponding weak $(p,p)$ inequality proved in \cite{CU-P99}, when $p >1$. Furthermore, if we take $u=w$, then both classes, $ A_{\Phi_{\varepsilon}}$ and $ A_1$, are equivalent. So we also obtain the well-known result related to the one weight weak $(1,1)$ type inequality of $T$ (\cite{javi}).

If $m >0$ and $v=1$, Theorem \ref{mixta conmutador (1,1)} supplies us with a two weight weak type inequality for commutators of CZOs that, as far as we know, has not been previously reported in the classical literature. If we also consider $u=w$, our result implies the corresponding estimate proved in \cite{P95}.

It is easy to see that the pair of weights $\left(u, M_{\Phi_{m+\varepsilon},v^{1-q}} u\right) \in A_{1}(v^{1-q})$. Therefore, Theorem \ref{mixta conmutador (1,1)} provides us with mixed weak inequalities of Fefferman-Stein type for $T_b^m$, $m \geq 0$, which have already been proved in \cite{ABP25}.

\section{Preliminaries and definitions}\label{section preliminares}

We shall understand a weight to be a locally integrable function $w$ that is positive almost everywhere. Given a weight $v$, the $A_{1}(v)$-class is defined as the collection of pairs of weights $(u,w)$ such that there exists a positive constant $C$ that verifies
\begin{equation}\label{claseA1}
    \frac{1}{v(Q)} \int_{Q}u(x)v(x)\,dx \leq C \inf_{Q}w
\end{equation} for every cube $Q \subset \mathbb{R}^{n}$ with sides parallel to the coordinate axes.
For $1<p< \infty$, the $A_{p}(v)$-class is the collection of pairs of weights $(u,w)$ that satisfy
\begin{equation}\label{claseAp}
    \left( \frac{1}{v(Q)} \int_{Q}u(x) v(x) \,dx \right) \left( \frac{1}{v(Q)} \int_{Q} w(x)^{1-p'}v(x)\,dx \right)^{p-1} \leq C
\end{equation} for some positive constant $C$ and for every cube $Q \subset \mathbb{R}^{n}$. When $v=1$, these are the well-known two weighted Muckenhoupt classes and we write $A_{p}(v)=A_p$. We shall denote $u \in A_{p}(v)$ to understand $(u,u)\in A_p(v)$. In that case, we say that $u \in A_{\infty}(v)$ if  $u \in A_p(v)$ for some $p \geq 1$. 

 Given $1<s<\infty$, we say that $v$ belongs to the reverse Hölder class $RH_s $ if there exists a positive constant $C$  such that  
\begin{equation}\label{RHs}
    \left( \frac{1}{|Q|} \int_{Q} v(x)^{s} \,dx  \right)^{1/s} \leq \frac{C}{|Q|} \int_{Q} v(x) \,dx
\end{equation} holds for every cube $Q \subset \mathbb{R}^{n}$. We also shall say that $v \in RH_{\infty}$ if there exists a positive constant $C$ such that the inequality 
\begin{equation}\label{RHinfty}
    \sup_{Q}v \leq \frac{C}{|Q|} \int_{Q}v(x) \,dx
\end{equation} holds for every $Q \subset \mathbb{R}^{n}$.

It can be shown that if $u \in A_p$, then $u$ belongs to an $RH_s$ class, for some $s>1$ (see for example \cite{javi} and \cite{grafakos}). Moreover, every $u \in A_p$ is doubling, that is there exists a positive constant $C$ such that
\begin{align*}
  u(2Q) \leq C u(Q)  
\end{align*}
for every cube $Q$.

The following propositions establish useful relations between classes of weights. Their proofs can be found in \cite{BCP19a} and \cite{CU-N-Reverse}, respectively.
\begin{propo}\label{pesos prop u A1 v en A infty u} If $u \in A_1$, then the following results are valid.
\begin{enumerate}[\rm i)]
\item \label{relacion u in av v in a infty u} If  $v \in A_{\infty}(u)$, then $u \in A_1 (v).$
\end{enumerate}
\end{propo}
\begin{propo}\label{v a la algo en A1}  If $v \in RH_{\infty}$, then $v^{\varepsilon} \in RH_{\infty}$ for $\varepsilon >0$ and $v^{1-q} \in A_{q}$ for some $1<q<\infty$. Moreover, if $v \in RH_{\infty}\cap A_{p}$ for $1 < p < \infty$, then $v^{1-p'} \in A_{1}$.
\end{propo}

We shall say that $\Phi:[0,\infty) \to [0,\infty)$ is a Young function if there exists a non-zero, non-negative and increasing function $\phi$ that verifies $\lim_{s\to \infty} \phi(s) = \infty$ and $ \displaystyle \Phi(t) = \int_{0}^{t} \phi(u) \,du$. It can be proved that $\Phi$ is continuous, convex and increasing, satisfying $\Phi(0)=0$ and $\displaystyle\lim_{t \rightarrow{\infty}} \Phi(t)/t = \infty$. We also consider $\Phi(t) = t$ as a Young function. Throughout this paper we shall write $\Phi_\varepsilon$ to mean the Young function $\Phi_{\varepsilon}(t)= t (1+\log^{+} t)^{\varepsilon}$, $\varepsilon \geq 0$.  

Given a Young function $\Phi$ we define the complementary function $\tilde{\Phi}$ by
\begin{align*}
    \tilde{\Phi}(t) = \sup \{ ts - \Phi(s) : s \geq 0 \}
\end{align*}
and the generalized inverse by
\begin{align*}
\Phi^{-1}(t) = \inf \{ s \geq 0 : \Phi(s) > t  \}.
\end{align*}

Given a Young function $\Phi$ and a weight $w$, we define the Luxemburg average of $f$ on a cube $Q$ with respect to $w$ by
\begin{equation*}
         \|f \|_{\Phi,Q,w}= \inf \left\{ \lambda >0 : \frac{1}{w(Q)} \int_{Q} \Phi\left(\frac{|f(x)|}{\lambda}\right)w(x) \,dx \leq 1 \right\}.
     \end{equation*} The generalized maximal operator associated to $\Phi$ and $w$ is defined by
     \begin{equation*}
         M_{\Phi, w}f(x) = \sup_{Q \ni x} \| f \|_{\Phi, Q, w},
     \end{equation*} where the supremum is taken over every cube $Q$ containing $x$. When $w=1$, we shall write $M_{\Phi,w}= M_{\Phi}$, the generalized maximal operator associated to $\Phi$. When $\Phi(t)=t$, we denote $M_{\Phi, w}=M_{w}$, the Hardy-Littlewood maximal operator with respect to the measure $d\mu(x) = w(x)dx$. When $\Phi(t)=t^{r}$, with $r>1$, then $M_{\Phi,w}f= M_{r,w}f=M_{w}(|f|^{r})^{1/r}$. Moreover, if  $k \in \mathbb{N}$, we shall also consider the iterated maximal operator $M^{k}_{w}$, which is the composition of $M_{w}$ with itself $k$ times.

Given a Young function $\Psi$ we say that $(u,w) \in A_{\Psi}(v)$ if
\begin{align*}
   \sup_{Q} \|u \|_{\Psi,Q,v} \| w^{-1} \chi_{Q}\|_{\infty} < \infty.
\end{align*} When $\Psi(t)=t$, $A_{\Psi}(v) = A_1 (v)$.

The following generalized Hölder inequality in the setting of Orlicz spaces will be useful through this paper. The proof can be found in \cite{ONeil65}.

\begin{teo}\label{Holder gen}
    Let $w$ be a doubling weight and $\Phi, \Psi $ and $ \Theta$ Young functions such that 
\begin{equation}\label{condicion triplete holder}
        \Phi^{-1}(t) \Psi^{-1}(t) \leq C \Theta^{-1}(t) \hspace{7mm} \text{for } t \geq t_0 \geq 0.
    \end{equation} Then, the inequality 
    \begin{equation*}
        \|fg \|_{\Theta, Q, w} \leq C \| f\|_{\Phi,Q,w} \|g\|_{\Psi,Q,w}
    \end{equation*} holds for every cube $Q \subset \mathbb{R}^n$.
\end{teo}

From the proof of the generalized Hölder inequality, it follows that if $\Phi, \Psi$ and $\Theta$ satisfy \eqref{condicion triplete holder}, then  \begin{align}\label{des Young generalizada}
    \Theta (ts) \leq \Phi(t) + \Psi(s)
\end{align} for $s, t \geq 0$.

 The proposition below allows us to establish certain order properties between two Young functions and their corresponding Luxemburg averages. Given two Young functions $\Phi$ and $\Psi$ we say that $\Phi$ dominates $\Psi$ at infinity, and we write $\Psi \prec_{\infty} \Phi$, if there exist positive constants $a,b$ and $t_0$ such that $\Psi(t) \leq b \, \Phi (at)$ for every $ t \geq t_0$. 
\begin{propo}\label{dominancia y prom lux}
Let $\Phi$ and $\Psi$ be Young functions such that $\Phi \prec_{\infty} \Psi$. Then there exists a positive constant $C$, depending on $\Phi$ and $\Psi$, such that for every $Q$ and $f$ it holds that  
    \begin{align*}
        \|f \|_{\Phi,Q,w} \leq C \|f \|_{\Psi,Q,w}.
    \end{align*}
\end{propo}
\begin{proof}
We assume that $\|f\|_{\Psi, Q,w} < \infty$, since in the other case the result is trivial. By hypotheses there exist $a,b$ and $t_{0}$ such that $\Phi(t) \leq b\Psi(at)$ for $t\geq  t_{0}$. 

Fix a cube $Q$ and consider the set
    \begin{align*}
        I = \{x \in Q : |f(x)| \leq a \| f\|_{\Psi,Q,w} \, t_0\}.
    \end{align*}
Then we get that
\begin{align*}
\begin{split}
    \frac{1}{w(Q)} \int_{Q} \Phi \left( \frac{|f(x)|}{a \| f\|_{\Psi,Q,w}} \right)w(x)\,dx & = \frac{1}{w(Q)} \int_{Q\cap I} \Phi \left( \frac{|f(x)|}{a \| f\|_{\Psi,Q,w}} \right) w(x)\,dx \\ & \hspace{10mm}+ \frac{1}{w(Q)} \int_{Q\backslash I} \Phi \left( \frac{|f(x)|}{a \| f\|_{\Psi,Q,w}} \right)w(x)\,dx \\
    &\leq \Phi(t_0) + \frac{b}{w(Q)} \int_{Q\backslash I} \Psi \left( \frac{|f(x)|}{ \| f\|_{\Psi,Q,w}} \right) w(x)\,dx\\ 
    & \leq \Phi(t_0) +b,
    \end{split}
\end{align*} since the last average is bounded by $1$.
If $\Phi(t_0) +b \leq 1$, it follows directly that
\begin{align*}
    \| f\|_{\Phi, Q,w} \leq a \| f\|_{\Psi,Q,w}.
\end{align*} If $\Phi(t_0) +b >1$, we obtain 
\begin{equation*}
    \| f\|_{\Phi, Q,w} \leq a(\Phi(t_0) +b ) \| f\|_{\Psi,Q,w}, 
\end{equation*}since $\Phi$ is convex.
\end{proof}

\begin{propo}\label{Young: igualdad normas a la r}
   Given $r>0$, a weight $w$ and a Young function $\Phi$, we have that \begin{equation*}
        \|f^{r} \|_{\Phi,Q,w} = \|f \|_{\Psi,Q,w}^{r},
    \end{equation*}
    where $\Psi(t)=\Phi(t^{r})$. Particularly, if $r \geq 1$ and $\Phi$ is such that $\Psi(t)=\Phi(t^{1/p})$ is also a Young function, $p\geq 1$, then
    \begin{equation*}
         \|f^{r/p} \|_{\Phi,Q,w} = \|f \|_{\Psi,Q,w}^{r/p},
    \end{equation*}
    where $\Psi(t)=\Phi(t^{r/p})$ is a $r$-Young function.
\end{propo}

\begin{proof}
For the first inequality, if we consider $\alpha=\lambda^{1/r}$ on the definition of the Luxemburg average, we get that
   \begin{align*}
   \begin{split}
    \|f^{r} \|_{\Phi,Q,w} &= \inf \left\{ \lambda >0 : \frac{1}{w(Q)}\int_{Q} \Phi\left( \frac{|f(x)|^{r}}{\lambda} \right)w(x)\,dx \leq 1 \right\} \\
    & = \inf \left\{ \lambda >0 : \frac{1}{w(Q)}\int_{Q} \Phi\left( \left( \frac{|f(x)|}{\lambda^{1/r}} \right)^{r}\, \right)w(x)\,dx \leq 1 \right\} \\
    & =\inf \left\{ \alpha^{r} >0 : \frac{1}{w(Q)}\int_{Q} \Psi\left( \frac{|f(x)|}{\alpha} \right)w(x)\,dx\leq 1 \right\} \\
    &=  \|f \|_{\Psi,Q,w}^{r}.
   \end{split}
   \end{align*}

To prove the second part, it is immediate that $ \Psi(t^{1/r}) = \Phi(t^{1/p})$ is a Young function. Using the first part with exponent $r/p$, it follows that
\begin{equation*}
 \| f^{r/p}\|_{\Phi,Q,w} = \|f \|_{\Psi,Q,w}^{r/p}. \qedhere
\end{equation*}
\end{proof}

The following lemma relates $M_{\Phi}$ and $M_{\Phi, w}$ when $w\in A_{1}$. The proof can be found in \cite{BCP22}.
\begin{lema}\label{lema comparacion puntual maximal phi con peso}
    Let $w\in A_{1}$ and $\Phi$ a Young function. Then there exists a positive constant $C>1$ such that
        \begin{align*}
            \|f \|_{\Phi,Q} \leq C \|f \|_{\Phi,Q,w}
        \end{align*} for every $f$. Moreover,  
        \begin{align*}
            M_{\Phi} f(x) \leq C M_{\Phi,w} f(x).
        \end{align*}
\end{lema}

\vspace{7mm}
A dyadic grid $\mathcal{D}$ is a collection of cubes $Q \subset \mathbb{R}^n$ such that
\begin{enumerate}[\rm i)]
    \item $\ell(Q) = 2^{k}$ for some $k \in \mathbb{Z}$,
    \item if $Q, P \in \mathcal{D}$ and $Q \cap P \neq \emptyset$, then $Q \subset P$ or $P\subset Q$,
    \item $D_k = \left\{ Q \in \mathcal{D} : \ell(Q) = 2^{-k} \right\}$ is a partition of $\mathbb{R}^n$ for each $k \in \mathbb{Z}$.
\end{enumerate} Associated to a dyadic grid and a weight $w$, we can define the dyadic maximal operator by 
\begin{align*}
    M_{\Phi,w}^{\mathcal{D}}f(x) = \mathop{\sup_{Q \ni x}}_{ Q \in \mathcal{D}} \|f \|_{\Phi,Q,w}.
\end{align*} It follows immediately that $M_{\Phi,w}^{\mathcal{D}}f(x) \leq M_{\Phi,w}f(x)$ almost everywhere. The next result can be found in \cite{HP-maximaldiadica}.
\begin{teo}\label{teorema maximal diadica}
    There exists dyadic grids $\mathcal{D}^{k}$, $ 1 \leq k \leq 3^n$, such that 
    \begin{align*}
        M f(x) \leq C \sum_{k=1}^{3^n} M^{\mathcal{D}^k}f(x).
    \end{align*}
\end{teo}

The following lemma provides a strong $(p,p)$ Fefferman-Stein type inequality for $M_{w}$.
\begin{lema}\label{lema: FS para maximal pesada}
  Let $1<p<\infty$, $u$ a non-negative function and $w$ a doubling weight. Then  \begin{align*}
        \int_{\mathbb{R}^n} |M_{w}f(x)|^{p} u(x) \, w(x)\,dx \leq C \int_{\mathbb{R}^n} |f(x)|^{p} M_{w} u(x) w(x)\,dx.
\end{align*}
\end{lema}
\begin{proof} We shall apply the Marcinkiewicz interpolation theorem with measures $d\mu(x) = M_w u(x)w(x) \,dx $ and $d\nu(x)= u(x)w(x) \,dx$. We first see that $M_{w}$ is bounded from $L^{\infty}(\mu)$ to $L^{\infty}(\nu)$. If $M_{w}u(x) = 0$ for some $x$, then $u = 0$ almost everywhere and the result is trivial. Suppose $M_{w}u(x)>0$ for all $x$, taking $\lambda > \| f\|_{\infty, \mu}$ we get that
\begin{align*}
   \int_{\{ x: |f(x)| > \lambda \} } \,d\mu(x) =  \int_{\{ x: |f(x)| > \lambda \} } M_{w}u (x) w(x) \,dx = 0.
\end{align*} Also, since  $M_{w}u(x)>0$ , then  $w \left(\{ x: |f(x)|>\lambda\} \right) =0.$ Consequently, $|f(x)|< \lambda$  $w-$c.t.p., and it follows that $M_{w}f \leq \lambda$, for every $\lambda > \|f \|_{\infty,\mu}$, and then $M_{w}f \leq \| f\|_{\infty,\mu}$. Therefore, 
\begin{align*}
    \| M_{w}f \|_{\infty, \nu} \leq \|f \|_{\infty, \mu}.
\end{align*}

On the other hand, the weak $(1,1)$ inequality follows from the fact that the pair $(u, M_{w}u) \in A_{1}(w)$ and therefore $M_{w}$ is bounded from $L(\mu)$ to $L(\nu)$ (see \cite{GCM01}).

To finalize the proof, we apply the Marcinkiewicz interpolation theorem and obtain the desired result.
\end{proof}

\begin{obs}\label{FS fuerte para maximal con diadica}
  It is easy to see that the above result is also valid for the dyadic maximal operator $M^{\mathcal{D}}_{w}$ under the same hypotheses on the pair of weights.
\end{obs}
\section{Auxiliary results}
In this section we give some auxiliary results needed in the proof of the main results. The first one is a known result for $BMO$ norms and it can be found, for example, in \cite{javi}.

\begin{lema}\label{obs norma bmo y bmo p}
    Let $1 < p <\infty$, then $\|\cdot \|_{\text{BMO}}$ is equivalent to
\begin{align*}
    \|b\|_{\text{BMO},p}= \sup_{Q} \left(\frac{1}{|Q|} \int_{Q} |b(x) - b_{Q}|^{p} \,dx \right)^{1/p}.
\end{align*}
\end{lema}

The following results can be found in \cite{P95}.
\begin{propo}\label{BMO valor abs con promedios en 2k Q}
    Given $b \in BMO$, there exists a positive constant $C$ such that 
\begin{align*}
    |b_Q - b_{2^{k}Q}| \leq C k \| b\|_{BMO}
\end{align*} for every $k \in \mathbb{N}$ and every cube $Q$.
\end{propo}

\begin{propo}\label{BMO comparacion norma phi y bmo}
   Let $b \in $ BMO, $\delta \geq 1$ and $\Psi(t)= e^{t^{1/\delta}}-1$. Then there exists a positive constant $C$ such that
   \begin{align*}
        \|b-b_{Q} \|_{\Psi,Q} \leq C \|b \|_{BMO}    \end{align*} holds for every cube $Q$.
\end{propo}

The next lemma establishes that a maximal operator asociated to a Young function is essencially constant if  can be found in \cite{anacomposition}.
\begin{lema}\label{FS:Lema 2.5 Ana}
Let $\Phi$ a Young function, $w$ a doubling weight, $f$ such that $M_{\Phi,w}f(x) < \infty$ almost everywhere, and a fixed cube $Q$. Then there exists a positive constant $C$, depending only on the dimension, such that  
    \begin{align*}
         M_{\Phi,w} (f\chi_{\mathbb{R}^{n}\backslash SQ})(x) \approx M_{\Phi,w} (f\chi_{\mathbb{R}^{n} \backslash SQ}) (y) ,
    \end{align*} for every $x,y \in Q$, where $S = 4 \sqrt{n}.$ 
\end{lema}
\begin{proof} Let $x,y \in Q$ and take a cube $Q_{0}$ containing $x$ such that $Q_{0} \cap \mathbb{R}^n \backslash SQ \neq \emptyset$. Denote $x_0$ and $x_{Q}$ the centres of $Q_0$ and $Q$ respectively, and consider $B_{0} = B(x_{0},\sqrt{n}\,l(Q_{0})/2)$. Since $Q_{0} \subset B_{0}$, we get that $B_{0} \cap \mathbb{R}^n \backslash SQ \neq \emptyset$. We shall prove that
\begin{align}\label{lema maximal infimo des auxiliar}
    \frac{3}{4} l(Q) \leq l(Q_{0}).
\end{align} Indeed, it that weren't true, taking $z \in B_{0}$ it follows that
\begin{align*}
    \begin{split}
        |z-x_Q| \leq |z-x| + |x-x_{Q}| \leq \sqrt{n}\, l(Q_{0}) + \frac{\sqrt{n}}{2} l(Q) \leq \frac{3}{4}\sqrt{n}\,l(Q)+ \frac{\sqrt{n}}{2}l(Q)  < \frac{S}{2} l(Q),
    \end{split}
\end{align*} which implies $B_{0} \subset SQ$ and contradicts the fact that $B_{0} \cap \mathbb{R}^n \backslash SQ \neq \emptyset$. Therefore, \eqref{lema maximal infimo des auxiliar} is valid. 

We now prove that $Q \subset S Q_{0}$. Let $z \in Q$, by \eqref{lema maximal infimo des auxiliar},
\begin{align*}
    \begin{split}
        |z-x_{0}| & \leq |z-x|+ |x-x_0 | \\
        & \leq \sqrt{n}\, l(Q) + \frac{\sqrt{n}}{2} l(Q_0) \\
        & \leq \frac{4}{3} \sqrt{n} \,l(Q_0 ) + \frac{\sqrt{n}}{2} l(Q_0) \\ & <\frac{S}{2}l(Q_0),
    \end{split}
\end{align*} which implies that $Q \subset B(x_{0}, S\, l(Q_0)/2 ) \subset S Q_{0} $. 

Since $w$ is doubling, it follows that
\begin{align*}
    \begin{split}
        \frac{1}{w(Q_0)} \int_{Q_0} \Phi \left( \frac{|f|\chi_{\mathbb{R}^n \backslash SQ}}{ \|f \chi_{\mathbb{R}^n \backslash SQ}  \|_{\Phi, SQ_{0}, w}}  \right)w   & \leq \frac{w(SQ_0)}{w(Q_0)} \frac{1}{w(SQ_0)} \int_{S Q_0} \Phi \left( \frac{|f|\chi_{\mathbb{R}^n \backslash SQ}}{ \|f \chi_{\mathbb{R}^n \backslash SQ}  \|_{\Phi, SQ_{0}, w}}  \right)w \\ & \leq C,
    \end{split}
\end{align*} which proves that $\|f \chi_{\mathbb{R}^n \backslash SQ} \|_{\Phi,Q_0 , w} \leq C \|f \chi_{\mathbb{R}^n \backslash SQ}  \|_{\Phi, SQ_{0}, w} \leq C M_{\Phi,w} (f\chi_{\mathbb{R}^n \backslash SQ})(y)$ for every $y \in Q \subset SQ_{0}$ and every cube $Q_0$ containing $x$. Then,
\begin{align*}
    M_{\Phi,w}(f\chi_{\mathbb{R}^n \backslash SQ})(x) \leq C M_{\Phi,w}(f\chi_{\mathbb{R}^n \backslash SQ})(y).
\end{align*}
 The other inequality follows similarly by interchanging the roles of $x$ and $y$.\qedhere
\end{proof}

A key tool to obtain mixed weak type inequalities for Calderón-Zygmund operators and their commutators consists in a strong Fefferman-Stein type inequality respect to a particular weight for the aforementioned operators. Concretely, in \cite{ABP25} the following result was obtained.
\begin{teo}\label{FS fuerte para conmutadores}
   Let $m\in \mathbb{N}\cup \{ 0\}$, $\varepsilon>0$, $q>1$ and $1<p<\min \{q,1+\frac{\varepsilon}{m+1}  \}$. If $v \in RH_{\infty} \cap A_{q'}$, then the inequality
\begin{equation}\label{FS: aux1, 1}
    \int_{\mathbb{R}^n} |T_{b}^{m}f(x)|^{p} w(x) v^{1-p}(x) \, dx \leq C \int_{\mathbb{R}^n} |f(x)|^{p} M_{\Phi_{m+\varepsilon}, v^{1-q}}w(x) v^{1-p}(x) \,dx,   \end{equation} holds for every non-negative and locally integrable function $w$.
\end{teo}

\section{Proof of main results}
\begin{proof}[Proof of Theorem \ref{teoremamaximal}.]
Without loss of generality, since $Mf = M(|f|)$ we can assume $f$ to be non-negative. Also, by Theorem \ref{teorema maximal diadica}, it suffices to prove this result with the dyadic maximal operator $M^{\mathcal{D}}$ associated to a dyadic grid $\mathcal{D}$.

Let us note that since $v\in RH_{\infty}$, the measure $\,d\mu(x) =v(x) \,dx$ is doubling $\mathbb{R}^n$. We consider the Calderón-Zygmund decomposition of $f$ at height $\lambda>0$ with respect to said measure and obtain a family of disjoint dyadic cubes $\left\{Q_j\right\}_{j}$ such that
    \begin{enumerate}[\rm i)]
        \item $f(x) \leq \lambda$ for almost every $x \notin \Omega = \bigcup_{j} Q_j$,
        \item $\displaystyle \lambda <\frac{1}{v(Q_j)} \int_{Q_j} f(x) v(x) \,dx < C\lambda $.
    \end{enumerate}

Let $\displaystyle f_{Q}^{v}=\frac{1}{v(Q_j)} \int_{Q_j} f(x) v(x) \,dx$, $Q_{j}^{*}=RQ_{j}$, with $R>1$, and $\Omega^{*}= \bigcup_{j} Q_{j}^{*}$. We write $f(x)=g(x) + h(x)$, where
 \begin{align*}
        g(x) =\left\{ \begin{array}{c c}
            f(x) & x \notin \Omega \\
            f_{Q_{j}}^{v} &  x \in Q_j
        \end{array}\right.
 \end{align*}
 and
 \begin{align*}
     h(x) = \sum_{j} \left( f(x) - f_{Q_{j}}^{v} \right) \chi_{Q_j}(x) = \sum_{j} h_j(x).
 \end{align*}
It follows directly that $g(x) \leq C \lambda$, for almost every $x \in \mathbb{R}^n $, and $ \int_{Q_j} h_{j} (x) v(x) \,dx =0$ for every $j$.

Given $x$ and $Q \in \mathcal{D}$ that contains $x$, we have that 
 \begin{align*}
 \begin{split}
     \frac{1}{|Q|} \int_{Q} f(y)v(y) \,dy & \leq \frac{1}{|Q|} \int_{Q} g(y)v(y) \,dx + \frac{1}{|Q|} \int_{Q} h(y)v(y) \,dx \\
     & \leq M^{\mathcal{D}}(gv)(x) + \sup_{Q \ni x} \left|  \frac{1}{|Q|} \int_{Q} h(y)v(y) \,dy  \right| .
  \end{split}\end{align*}
Taking supremum on the left hand side over the dyadic cubes in $\mathcal{D}$ that contain $x$, it follows that 
\begin{align*}
    M^{\mathcal{D}}(fv)(x) \leq M^{\mathcal{D}}(gv) (x) + \sup_{Q \ni x}  \left|  \frac{1}{|Q|} \int_{Q} h(y)v(y) \,dy  \right| .
\end{align*} Let
\begin{align*}
    \Tilde{M}^{\mathcal{D}}f(x) = \sup_{Q \ni x} \left|\frac{1}{|Q|} \int_{Q} f(y) \,dy\right|,
\end{align*} therefore
\begin{align*}
    \begin{split}
         uv\left( \left\{ x \in \mathbb{R}^n : \frac{M^{\mathcal{D}}(fv)(x)}{v(x)} > \lambda \right\} \right) & \leq  uv\left( \left\{ x \in \mathbb{R}^n \backslash \Omega^{*}: \frac{M^{\mathcal{D}}(gv)(x)}{v(x)} > \lambda /2 \right\} \right)  +uv(\Omega^{*}) \\ & \hspace{10mm} +  uv\left( \left\{ x \in \mathbb{R}^{n} \backslash \Omega^{*} : \frac{\tilde{M^{\mathcal{D}}}(hv)(x)}{v(x)} > \lambda/2 \right\} \right) \\
         & = I_{1} + I_{2} + I_{3}.
    \end{split}
\end{align*}

To estimate $I_1$, let us note first that Lemma \ref{v a la algo en A1} and Lemma \ref{lema comparacion puntual maximal phi con peso} imply that
 \begin{align}\label{aux M (1,1) 1}
    M^{\mathcal{D}}f(x) \leq M^{\mathcal{D}}_{v^{{1-q}}}f(x) .
 \end{align} By Observation \ref{FS fuerte para maximal con diadica}, we can apply Lemma \ref{lema: FS para maximal pesada} to the dyadic maximal operator. By Tchebyshev inequality, \eqref{aux M (1,1) 1} and Lemma \ref{lema: FS para maximal pesada} we get that
\begin{align*}
    I_{1} & \leq \frac{1}{\lambda^{q}} \int_{\mathbb{R}^n} M^{\mathcal{D}}(gv)(x)^{q} u^{*}(x)v(x)^{1-q} \,dx \\
    &\leq  \frac{C}{\lambda^{q}} \int_{\mathbb{R}^n} M^{\mathcal{D}}_{v^{1-q}}(gv)(x)^{q} u^{*}(x) v(x)^{1-q} \,dx\\
     & \leq \frac{C}{\lambda^q} \int_{\mathbb{R}^n} |g(x)|^q M^{\mathcal{D}}_{v^{1-q}}u^{*} (x) v(x) \,dx,\end{align*} where we have used that $v^{1-q} \in A_{1}$ and $u^{*}= u \chi_{\mathbb{R}^{n} \backslash \Omega^{*}} $. By the properties of $g$, it follows that
     \begin{align*}
         \begin{split}
             I_{1} & \leq  \frac{C}{\lambda} \int_{\mathbb{R}^n } |g(x)| M^{\mathcal{D}}_{v^{1-q}}u^{*}(x) v(x) \,dx \\
     & \leq  \frac{C}{\lambda} \left( \int_{\mathbb{R}^n \backslash \Omega} f(x)M^{\mathcal{D}}_{v^{1-q}}u^{*}(x) v(x) \,dx + \sum_{j}  \frac{1}{v(Q_j)} \int_{Q_{j}} f(x) v(x) \,dx \int_{Q_j} M^{\mathcal{D}}_{v^{1-q}}u^{*}(x) v(x) \,dx\right).
         \end{split}
     \end{align*}
 
Since $(u,w) \in A_{1}(v^{1-q})$, we get that $M^{\mathcal{D}}_{v^{1-q}}u^{*} (x)\leq M_{v^{1-q}}u^{*}(x) \leq M_{v^{1-q}}u(x) \leq C \inf_{Q} w$ for almost every $x \in \mathbb{R}^{n}$, which allows us to estimate the left summand by the desired result. On the other hand, by Lemma \ref{FS:Lema 2.5 Ana} we have that
\begin{align*}\begin{split}
    \frac{1}{v(Q_j)}  \int_{Q_j} f(x) v(x) \,dx \int_{Q_j} M^{\mathcal{D}}_{v^{1-q}}u^{*}(x) v(x)\,dx & \leq \frac{1}{v(Q_j)}  \int_{Q_j} f(x) v(x) \,dx \int_{Q_j} M_{v^{1-q}}u^{*}(x) v(x)\,dx \\ & \leq C  \frac{1}{v(Q_j)}  \int_{Q_j} f(x) v(x) \,dx\,\, v(Q_j) \,\inf_{Q_j} M_{v^{1-q}}u^{*} \\ 
   & \leq C  \int_{Q_j} f(x) v(x) M_{v^{1-q}}u^{*}(x)\,dx \\ &\leq C \int_{Q_j} f(x) w(x) v(x) \,dx.
\end{split}\end{align*}
Since the cubes $Q_j$ are disjoint, finally we get that  
\begin{equation*}
    I_{1} \leq \frac{C}{\lambda}  \int_{\mathbb{R}^n} f(x) w(x) v(x) \,dx.
\end{equation*}

Let us now estimate $I_2$. Proposition \ref{v a la algo en A1} implies that $v^{q} \in RH_{\infty}$ and $v^{1-q} \in A_{1}$, respectively. Applying the hypothesis on the pair $(u,w)$, 
\begin{equation*}
\begin{split}
     I_{2} & \leq \sum_{j}  uv^{1-q}v^{q}\left(Q_{j}^{*} \right) \\
     & \leq \sum_{j}v^{1-q}(Q_{j}^{*}) \sup_{Q_{j}^{*}} v^{q} \frac{1}{v^{1-q}(Q_{j}^{*})} \int_{Q_{j}^{*}} u(x) v(x)^{1-q}\,dx \\
     & \leq C \sum_{j} v^{1-q}(Q_{j}^{*}) \sup_{Q_{j}^{*}} v^{q} \, \, \inf_{Q_{j}^{*}} w \\
     &\leq C \sum_{j}  \frac{v^{1-q}(Q_{j}^{*})}{|Q_{j}^{*}|} \int_{Q_{j}^{*}} v(x)^{q} \,dx \, \, \inf_{Q_{j}^{*}} w  \\
     & \leq C \sum_{j} \inf_{Q_{j}^{*}} v^{1-q} \int_{Q_{j}^{*}} v(x)^{q} \,dx \,\, \inf_{Q_{j}^{*}} w\\
     &  \leq C \sum_{j}\int_{Q_{j}^{*}} v(x) \,dx \,\, \inf_{Q_{j}^{*}} w  \leq C \sum_{j} \int_{Q_j} v(x) \,dx\,\, \inf_{Q_{j}^{*}} w
\end{split}
\end{equation*}
and, by the Calderón-Zygmund decomposition, we get that  
\begin{align*}
    \begin{split}
         I_{2} & \leq \frac{C}{ \lambda} \sum_{j}  \int_{Q_j} f(x)  v(x) \,dx \,\,\inf_{Q_{j}^{*}} w \\
         & \leq \frac{C}{ \lambda} \sum_{j}  \int_{Q_j} f(x) w(x) v(x) \,dx \\
         & \leq \frac{C}{\lambda}\int_{\mathbb{R}^n}  f(x) w(x) v(x) \,dx,
    \end{split} 
\end{align*}
which is the desired estimate.

At last, we prove that $I_3=0$. Indeed, let $x \in \mathbb{R}^{n} \backslash \Omega^*$ and $Q$  a dyadic cube that contains $x$ and intersects $\Omega$, then we either have that $Q_j \subset Q$, or $Q_j \cap Q = \emptyset $. By the property of $h$, it follows that the average 
\begin{align*}
    \left| \frac{1}{|Q|} \int_{Q} h(y) v(y) \,dy \right| = \frac{1}{|Q|} \left|\sum_{j: Q_j \subset Q} \int_{Q_j \cap Q} h(y) v(y) \,dy \right|= 0.
\end{align*}
Therefore, $ \tilde{M}^{\mathcal{
D}}(hv)(x) = 0$ for almost every $x \in \mathbb{R}^n \backslash \Omega^{*}$, from where we get that $I_3 =0.$ 
\end{proof}

\begin{proof}[Proof of Theorem \ref{mixta conmutador (1,1)}]
Without loss of generality, we can assume $f$ non-negative and with compact support. 

As in the previous theorem, we consider the Calderón-Zygmund decomposition of $f$ at height $\lambda >0$ respect to the measure $d\mu(x)=v(x) \,dx$ and obtain a family of disjoint dyadic cubes $\{ Q_j \}_{j}$ such that
    \begin{enumerate}[\rm i)]
        \item $f(x) \leq \lambda$ for almost every $x \notin \Omega= \bigcup_{j} Q_j$,
        \item $\displaystyle \lambda <\frac{1}{v(Q_j)} \int_{Q_j} f(x) v(x) \,dx < C\lambda $.
    \end{enumerate}

   We denote $\displaystyle f_{Q}^{v}=\frac{1}{v(Q_j)} \int_{Q_j} fv $, $Q_{j}^{*}=4\sqrt{n} \,Q_j$, $\Omega^{*}= \bigcup_{j} Q_{j}^{*}$, $x_j$ the center of $Q_j$ and $l_j$ its side length.  As before, we write $f(x)=g(x) + h(x)$, where
    \begin{align*}
        g(x) =\left\{ \begin{array}{c c}
            f(x) & x \notin \Omega \\
            f_{Q_{j}}^{v} &  x \in Q_j
        \end{array}\right.
    \end{align*}
and  
 \begin{align*}
     h(x) = \sum_{j} \left( f(x) -   f_{Q_{j}}^{v} \right) \chi_{Q_j}(x) = \sum_{j} h_j(x).
 \end{align*}

We first prove the case $m=0$. It follows that
\begin{align*}
    \begin{split}
       uv\left( \left\{ x \in \mathbb{R}^n : \frac{|T(fv)(x)|}{v(x)} > \lambda \right\} \right) & \leq uv\left( \left\{ x \in \mathbb{R}^n \backslash \Omega^{*} : \frac{|T(gv)(x)|}{v(x)} >  \frac{\lambda}{2} \right\} \right)+ uv(\Omega^{*}) \\
       & \hspace{10mm} +uv\left( \left\{ x \in \mathbb{R}^n \backslash \Omega^{*}: \frac{|T(hv)(x)|}{v(x)} > \frac{\lambda}{2}\right\} \right) \\
       & = I_{1} + I_{2} + I_{3}.
    \end{split}
\end{align*}

To estimate $I_1$, consider $u^{*}= u \chi_{\mathbb{R}^n \backslash \Omega^{*}}$, then, by Tchebyshev inequality and Theorem \ref{FS fuerte para conmutadores} with $1<s< \min \{q, 1 +\varepsilon \}$, we get that
\begin{align*}
\begin{split}
        I_1 & \leq \frac{C}{\lambda^{s}} \int_{\mathbb{R}^n} |T(gv)(x)|^{s} u^{*}(x)v(x)^{1-s} \,dx \\
& \leq \frac{C}{\lambda^s} \int_{\mathbb{R}^n} |g(x)|^s M_{\Phi_{\varepsilon}, v^{1-q}}u^{*}(x) v(x) \,dx \\
     & \leq \frac{C}{\lambda} \int_{\mathbb{R}^n} |g(x)| M_{\Phi_{\varepsilon}, v^{1-q}}u^{*}(x) v(x) \,dx.\end{split}
   \end{align*} Proceeding as in the proof of Theorem \ref{teoremamaximal} we obtain the desired estimate for $I_1.$

Let us note that since $t \leq \Phi_{\varepsilon}(t)$, the hypotheses on the weights imply that $(u,w) \in A_{1}(v^{1-q})$, which is the condition of Theorem \ref{teoremamaximal}. Therefore, we can follow the steps made in the proof to obtain the desired estimate for $I_2$.

Finally, we estimate the term $I_3$. Since each $h_j$ is supported in $Q_j$, given $x \in \mathbb{R}^{n} \backslash \Omega^{*}$ we have that $T(h_{j}v)(x)$ admits the integral representation \eqref{CZO definicion integral}. Therefore,
\begin{align*}
\begin{split}
     I_3 & \leq   uv\left( \left\{ x \in \mathbb{R}^{n} \backslash \Omega^{*} : \frac{|T(\sum_{j}h_{j}v)(x)|}{v(x)} > \lambda /2 \right\} \right)  \\
     & \leq uv\left( \left\{ x \in \mathbb{R}^{n} \backslash \Omega^{*} : \frac{  \sum_{j} |T(h_{j}v)(x)|}{v(x)} > \lambda /2 \right\} \right) \\
     & \leq \frac{C}{\lambda}  \int_{\mathbb{R}^n \backslash \Omega^{*}}  \sum_{j}|T(h_{j}v)(x)| u(x) \,dx \\
     & = \frac{C}{\lambda}  \int_{\mathbb{R}^n \backslash \Omega^{*}} \sum_{j} \int_{Q_j} |K(x-y) - K(x-x_j)| |h_{j}(y)| v(y) \,dy \, u(x) \,dx.
\end{split}
\end{align*}
Considering that $x \in \mathbb{R}^{n}\backslash \Omega^{*}$ and $y \in Q_j$, it follows that $ |x-x_j| > 2 \sqrt{n}l_j  \geq 2 |x_{j}-y|$. By condition \eqref{condicion de suavidad} and Tonelli theorem, we get that
\begin{align*}
    \begin{split}
       I_3 & \leq\frac{C}{\lambda} \sum_{j}  \int_{\mathbb{R}^n \backslash Q_{j}^{*}} \int_{Q_j} \frac{|y-x_{j}|}{|x-x_{j}|^{n+1}} |h_{j}(y)| v(y) \,dy \, u(x) \,dx  \\
     & \leq \frac{C}{\lambda} \sum_{j}  \int_{Q_j} |y-x_{j}| |h_{j}(y)| v(y) \,dy  \int_{\mathbb{R}^{n}\backslash Q^{*}_{j}} |x-x_{j}|^{-n-1} u(x) \,dx\\&
     \leq \frac{C}{\lambda} \sum_{j}  \int_{Q_j} l_{j} |h_{j}(y)| v(y) \,dy  \int_{\mathbb{R}^{n}\backslash Q^{*}_{j}} |x-x_{j}|^{-n-1} u(x) \,dx.
    \end{split}
\end{align*}
Since $\mathbb{R}^n \backslash Q_{j}^{*} \subset \mathbb{R}^n \backslash B(x_{j},2 \sqrt{n}\,l_j) ,
$ we define the sets
 \begin{align}\label{aux coronas demostracion OCZ}
     A_{j,k}=\{ x : \sqrt{n}\, l_j \, 2^{k} < |x-x_{j}| \leq \sqrt{n}\, l_j \, 2^{k+1} \},
 \end{align} and obtain that
\begin{align*}
\begin{split}
    \int_{\mathbb{R}^{n} \backslash Q^{*}_j}  |x-x_{j}|^{-n-1} u(x) \,dx & \leq \int_{\mathbb{R}^{n} \backslash B(x_{j}, 2\sqrt{n} l_j)} \frac{1}{|x-x_{j}|^{n+1}} u(x) \,dx  \\
    &= \sum_{k=1}^{\infty} \int_{A_{j,k} } \frac{1}{|x-x_{j}|^{n+1}} u(x) \,dx \\
    & \leq C \sum_{k=1}^{\infty}  \frac{1}{(\sqrt{n} l_j 2^{k-1})^{n+1}} \int_{B(x_{j}, 2^{k} \sqrt{n} l_{j})} u(x) \,dx 
    \end{split}\end{align*}
    
   \begin{align*}
       \begin{split}
   \textcolor{white}{ \int_{\mathbb{R}^{n} \backslash Q^{*}_j}  |x-x_{j}|^{-n-1} u(x) \,dx } & \leq C l_{j}^{-1}   \sum_{k=1}^{\infty} 2^{-k} \frac{1}{(\sqrt{n} l_j 2^{k})^{n}}
\int_{B(x_{j}, 2^{k} \sqrt{n} l_j)} u(x)\,dx \\
    & \leq C l_{j}^{-1}\sum_{k=1}^{\infty}  2^{-k} \|u\|_{\Phi_{\varepsilon}, 2^{k+1}\sqrt{n}Q_{j}}\\
    & \leq C l_{j}^{-1} \|u\|_{\Phi_{\varepsilon}, 2^{k+1}\sqrt{n}Q_{j},v^{1-q}}\\
& \leq C l_{j}^{-1}\inf_{Q_j} w,
\end{split}\end{align*}
where we have used Propositions \ref{dominancia y prom lux} and \ref{lema comparacion puntual maximal phi con peso} and the hypotheses on $(u,w)$.

Going back to the estimation of $I_{3}$, definition of $h_j$ allows us to get that  
\begin{align*}
\begin{split}
    I_{3} & \leq \frac{C}{\lambda} \sum_{j} \inf_{Q_j} w\, \int_{Q_j}  |h_{j}(y)| v(y) \,dy 
    \\ &= \frac{C}{\lambda} \sum_{j} \inf_{Q_j} w \, \int_{Q_j} |f(y) - f_{Q_j}^{v}| v(y) \,dy \\
    & \leq \frac{C}{\lambda} \left( \sum_{j}\inf_{Q_j} w\, \int_{Q_j} f(y) v(y) \,dy + \sum_{j} \inf_{Q_j} w\, \frac{1}{v(Q_j)} \int_{Q_j} f(y)v(y) \,dy \,\,v(Q_j)   \right) \\
    & \leq \frac{C}{\lambda}  \left( \sum_{j} \int_{Q_j} f(y) v(y) w(y) \,dy + \sum_{j}   \int_{Q_j} f(y)v(y) w(y) \,dy  \right) \\
    & \leq  \frac{C}{\lambda} \int_{\mathbb{R}^n} f(y) w(y) v(y) \,dy,
\end{split}
\end{align*}
which completes the case $m=0$.

 Let us now consider the case $m=1$. Without loss of generality, we take $b \in BMO$ such that $\| b\|_{\text{BMO}}=1$, since $T_{b/\| b\|_{\text{BMO}}}f = T_{b} (f/\|b \|_{\text{BMO}})$. The Calderón-Zygmund decomposition used at the beginning of this proof yields
    \begin{align*}
\begin{split}
    uv\left( \left\{ x \in \mathbb{R}^n : \frac{|T_{b}(fv)(x)|}{v(x)} > \lambda \right\} \right) &\leq  uv\left( \left\{ x \in \mathbb{R}^n \backslash \Omega^*: \frac{|T_{b}(gv)(x)|}{v(x)} > \frac{\lambda}{2} \right\} \right) + uv(\Omega^{*}) \\
    & \hspace{5mm} + uv\left( \left\{ x \in \mathbb{R}^n \backslash{\Omega^{*}}: \frac{|T_{b}(hv)(x)|}{v(x)} >\frac{\lambda}{2} \right\} \right)\\
    & = I^{1}_{1}+I^{1}_{2}+I^{1}_{3}.
\end{split}
\end{align*} 

We begin by estimating $I^{1}_1$. Let  $u^{*}=u \,\chi_{\mathbb{R}^{n}\backslash\Omega^{*}}$, by Tchebyshev inequality with $1 < p < \min \{q, 1+\frac{\varepsilon}{2} \} $ and Theorem \ref{FS fuerte para conmutadores}, we get that
\begin{align*}
    \begin{split}
        I^{1}_{1} & \leq \frac{C}{\lambda^{p}} \int_{\mathbb{R}^n} |T_{b}(gv)(x)|^{p} u^{*}(x)v(x)^{1-p} \,dx \\
& \leq \frac{C}{\lambda^p} \int_{\mathbb{R}^n} |g(x)|^p M_{\Phi_{1+\varepsilon},v^{1-q}}u^{*}(x) v(x) \,dx  \\
& \leq \frac{C}{\lambda} \int_{\mathbb{R}^n} |g(x)| M_{\Phi_{1+\varepsilon},v^{1-q}}u^{*}(x) v(x) \,dx.
    \end{split}
\end{align*}
Proceeding as in the estimation of $I_1$, from the definition of $g$ it follows that 
\begin{align*}
    \begin{split}
    I_{1}^{1} &\leq \frac{C}{\lambda} \left( \int_{\mathbb{R}^n \backslash \Omega} f(x) M_{\Phi_{1+\varepsilon}, v^{1-q}}u^{*}(x) v(x) \,dx + \sum_{j}  f^{v}_{Q_j} \int_{Q_j} M_{\Phi_{1+\varepsilon}, v^{1-q}}u^{*}(x) v(x) \,dx \right).  
    \end{split}
\end{align*}

Since $u^*$ is supported in the complement of $\Omega^{*}$, applying Lemma \ref{FS:Lema 2.5 Ana} with $y \in Q_j$, we get that
\begin{align*}
    M_{\Phi_{1+\varepsilon}, v^{1-q}}u^{*}(y) \approx\inf_{ Q_j} M_{\Phi_{1+\varepsilon}, v^{1-q}}u^{*}
\end{align*}
and for every $Q_j$ we obtain
\begin{align*}
 \begin{split}
         f^{v}_{Q_j} \int_{Q_j} M_{\Phi_{1+\varepsilon}, v^{1-q}}u^{*}(x) v(x) \,dx &  \leq C f^{v}_{Q_j} \,v(Q_j) \,\inf_{Q_j}  M_{\Phi_{1+\varepsilon}, v^{1-q}}u^{*} \\
         & \leq \int_{Q_{j}} f(x) M_{\Phi_{1+\varepsilon}, v^{1-q}}u^{*}(x) v(x) \,dx.
 \end{split}
\end{align*}
Therefore,
\begin{align*}
    \begin{split}
       I_{1}^{1}& \leq \frac{C}{\lambda}  \int_{\mathbb{R}^{n}\backslash \Omega} f(x) M_{\Phi_{1+\varepsilon}, v^{1-q}}u^{*}(x) v(x) \,dx  +\frac{C}{\lambda} \sum_{j}   \int_{Q_{j}} f(x) M_{\Phi_{1+\varepsilon}, v^{1-q}}u^{*}(x) v(x) \,dx  \\ 
        & \leq C \int_{\mathbb{R}^{n}} \frac{f(x)}{\lambda} M_{\Phi_{1+\varepsilon}, v^{1-q}}u(x) v(x) \,dx \\
        & \leq C\int_{\mathbb{R}^n} \Phi_{1} \left(\frac{|f(x)|}{\lambda} \right)   w (x) v(x) \,dx, 
    \end{split}
\end{align*} where at last we have used that $ t \leq \Phi_{1} (t)$ and the hypotheses on $(u,w)$.

The estimation of $I_{2}^{1}$ follows as in the estimation of $I_2,$ observing that the hypotheses on the weights imply $(u,w) \in A_{1}(v^{1-q})$.

We now estimate the term $I_{3}^{1}$. Since 
\begin{align*}
    T_{b}(hv)(x)= \sum_{j} (b(x)-b_{Q_j}) T(h_{j}v)(x)- \sum_{j} T((b(x)-b_{Q_j})h_{j}v)(x),
\end{align*}
we get that 
\begin{align*}
    \begin{split}
        I_{3}^{1} &\leq  uv\left( \left\{ x \in \mathbb{R}^n\backslash \Omega^{*} : \left|\sum_{j} \frac{(b-b_{Q_j}) T(h_{j}v)(x)}{v(x)}\right| > \frac{\lambda}{4}    \right\} \right) \\ & \hspace{5mm}+  uv\left( \left\{ x \in \mathbb{R}^n\backslash \Omega^{*} : \left|\sum_{j} \frac{T((b-b_{Q_j})h_{j}v)(x)}{v(x)} \right| > \frac{\lambda}{4}   \right\} \right) \\
        & = I_{3,1}^{1} + I_{3,2}^{1}.
    \end{split}
\end{align*} 

 For the term $I_{3,1}^{1}$, by Tchebyshev inequality and \eqref{CZO definicion integral} we obtain
\begin{align*}
    \begin{split}
        I_{3,1}^{1} &\leq \frac{C}{\lambda} \sum_{j}\int_{\mathbb{R}^n \backslash Q_{j}^{*}}|b(x)-b_{Q_j}| T(h_{j}v)(x)u_{j}^{*}(x) \,dx \\
        & \leq \frac{C}{\lambda} \sum_{j}\int_{\mathbb{R}^{n} \backslash Q_{j}^{*}} |b(x)-b_{Q_j}| \left| \int_{Q_j} h_{j}(y)v(y) (K(x-y)-K(x-x_j)) \, dy \right| u^{*}_{j}(x) \,dx\\
        & \leq \frac{C}{\lambda} \sum_{j} \int_{Q_j} |h_{j}(y)| v(y) \int_{\mathbb{R}^{n} \backslash Q_{j}^{*}} |b(x)-b_{Q_j}| |K(x-y)-K(x-x_j) | u^{*}_{j}(x) \,dx \, dy,
    \end{split}
\end{align*} where $u_{j}^{*} = u \chi_{Q_{j}^{*}}$.

Recall the sets $A_{j,k}$ defined in  \eqref{aux coronas demostracion OCZ}, therefore by condition \eqref{condicion de suavidad} we can estimate the inner integral by \begin{align*}
    \begin{split}
        \int_{\mathbb{R}^{n} \backslash Q_{j}^{*}} |b(x)-b_{Q_j}| & |K(x-y)-K(x-x_j) | u^{*}_{j}(x) \,dx \\ &= \sum_{k=1}^{\infty} \int_{A_{j,k}} |b(x)-b_{Q_j}|  |K(x-y)-K(x-x_j) | u^{*}_{j}(x) \,dx \\& \leq \sum_{k=1}^{\infty}\int_{A_{j,k}}  |b(x)-b_{Q_j}| \frac{|y-x_j|}{|x-x_j|^{n+1}} u^{*}_{j}(x) \,dx \\
        &  \leq C  \sum_{k=1}^{\infty} \frac{l_j}{\sqrt{n}l_j 2^k} \frac{1}{(\sqrt{n}l_j 2^{k+1})^n} \int_{B(x_j, \sqrt{n}l_j 2^{k+1})} |b(x)-b_{Q_j}| u^{*}_{j}(x) \,dx  \\
        &\leq C \sum_{k=1}^{\infty}  \frac{2^{-k}}{|\sqrt{n} 2^{k+2} Q_j |} \int_{\sqrt{n} 2^{k+2} Q_j } |b(x)-b_{Q_j}| u^{*}_{j}(x) \,dx.
    \end{split}
\end{align*}

Let $k_0$ be the only integer such that $2^{k_0 -1} \leq \sqrt{n} < 2^{k_0}$. Applying Proposition \ref{BMO valor abs con promedios en 2k Q}, the generalized Hölder inequality and Proposition \ref{BMO comparacion norma phi y bmo} we get that
\begin{align*}
    \begin{split}
    \frac{1}{|\sqrt{n}  2^{k+2} Q_j |} \int_{\sqrt{n}  2^{k+2} Q_j } |b(x)-b_{Q_j}|& u^{*}_{j}(x) \,dx \\& \leq     
    \frac{C}{|2^{k+k_0 +2} Q_j|} \int_{ 2^{k+k_0 +2} Q_j} |b(x)-b_{ 2^{k+k_0 +2} Q_j}| u^{*}_{j}(x) \,dx \\ & \hspace{7mm}+ C (k+ k_0 +2) M u^{*}_{j}(y) \\& \leq  C \| b-b_{ 2^{k+k_0 +2} Q_j} \|_{\tilde{\Phi}_{1+\varepsilon}, \sqrt{n}  2^{k+ k_0 +2} Q_j} \|u^{*}_{j} \|_{\Phi_{1+\varepsilon}, \sqrt{n}  2^{k+ k_0 +2} Q_j} \\
    & \hspace{7mm} + C (k+ k_0 +2) M u^{*}_{j}(y) \\ & \leq C M_{\Phi_{1+\varepsilon}} u^{*}_{j}(y). 
    \end{split}
\end{align*}

Therefore, Lemma \ref{FS:Lema 2.5 Ana} and the hypotheses on $(u,w)$ allow us to obtain  
\begin{equation*}\label{como en I3,1 del teo 222}
    \begin{split}
        I_{3,1}^{1}& \leq \frac{C}{\lambda} \sum_{j} \int_{Q_j} |h_{j}(y)| M_{\Phi_{1+\varepsilon}} u^{*}_{j}(y) v(y) \,dy \\ & \leq \frac{C}{\lambda} 
\sum_j \inf_{Q_j} M_{\Phi_{1+\varepsilon}}u^{*}_{j} \,\, \int_{Q_j} fv + C \sum_{j} \inf_{Q_j} M_{\Phi_{1+\varepsilon}}u^{*}_{j} \,\, \int_{Q_j} f^{v}_{Q_j}v \\
& \leq \frac{C}{\lambda} \int_{\mathbb{R}^n} f(x) M_{\Phi_{1+\varepsilon}}u(x) v(x)\,dx + \frac{C}{\lambda} \sum_{j}  \inf_{Q_j} M_{\Phi_{1+\varepsilon}}u^{*}_{j} \,\,\int_{Q_j} fv \\ & \leq C \int_{\mathbb{R}^n} \frac{f(x)}{\lambda}  w(x) v(x)\,dx \\
& \leq C\int_{\mathbb{R}^n} \Phi_{1} \left(\frac{|f(x)|}{\lambda} \right)  w (x) v(x) \,dx.
    \end{split}
\end{equation*}

 Now we estimate $I_{3,2}^{1}$. Applying the case $m=0$ with the pair $(u^{*},  M_{\Phi_{\varepsilon}, v^{1-q}}u^{*})$ and Lemma \ref{FS:Lema 2.5 Ana}, it follows that
\begin{align*}
    \begin{split}
I_{3,2}^{1} &= u^{*}v\left( \left\{ x \in \mathbb{R}^n : \left|\sum_{j} \frac{T((b-b_{Q_j})h_{j}v)(x)}{v(x)}\right| > \frac{\lambda}{4}   \right\} \right)  \\
& \leq  u^{*}v\left( \left\{ x \in \mathbb{R}^n : \left| \frac{ T \left(\sum_{j} (b-b_{Q_j})h_{j}v\right)(x)}{v(x)} \right| > \frac{\lambda}{4}   \right\} \right) \\
&\leq  \frac{C}{\lambda}\int_{\mathbb{R}^n} \left| \sum_{j} (b(x)-b_{Q_j}) h_{j}(x) \right| M_{\Phi_{\varepsilon}, v^{1-q}}u^{*}(x) v (x) \, dx\\
& \leq  \frac{C}{\lambda} \sum_{j} \int_{Q_j} |b(x)-b_{Q_j}| |h_{j}(x)| M_{\Phi_{\varepsilon}, v^{1-q}}u^{*}(x) v (x) \, dx\\ & \leq \frac{C}{\lambda}  \sum_{j}\inf_{Q_j}M_{\Phi_{\varepsilon}, v^{1-q}}u^{*} \int_{Q_j} |b(x)-b_{Q_j}| |h_{j}(x)| v (x) \, dx.
    \end{split}
\end{align*}

Let us note that, by definition of $h_j$,
\begin{align*}
    \begin{split}
      \int_{Q_j} |b(x)-b_{Q_j}| |h_{j}(x)|  v (x) \, dx & \leq  \int_{Q_j} |b(x) - b_{Q_j}| f(x) v(x) \, dx + f_{Q_{j}}^{v} \int_{Q_j} |b(x) - b_{Q_j}| v(x) \,dx  
    \end{split}
\end{align*}
and, since $v \in RH_{\infty}$, we get that 
\begin{align*}
    \begin{split}
        f_{Q_{j}}^{v} \int_{Q_j} |b(x) - b_{Q_j}| v(x) \,dx &\leq C \frac{v(Q_j)}{|Q_j|} \frac{1}{v(Q_j)} \int_{Q_j} f(x)v(x)\,dx \int_{Q_j} |b(x)-b_{Q_j}|\,dx \\
        & \leq C\int_{Q_j}f(x)v(x)\,dx .
    \end{split}
\end{align*}
On the other hand, by the generalized Hölder inequality (Teorema \ref{Holder gen}) with $\Phi_{1}$ and $\tilde{\Phi}_{1}$ and measure $d\mu(x)  = v(x) \,dx $, we obtain
\begin{align*}
    \begin{split}
        \int_{Q_j} |b(x) - b_{Q_j}| f(x) v(x) \, dx & \leq C v(Q_j)  \|b-b_{Q_j} \|_{\tilde{\Phi}_{1}, Q_{j},v} \|f \|_{\Phi_{1}, Q_j,v} \\
        & \leq C v(Q_j) \|f \|_{\Phi_{1}, Q_j,v} \\ & \leq  Cv(Q_j) \left( \lambda +\frac{\lambda}{v(Q_j)} \int_{Q_j} \Phi_{1} \left(\frac{|f(x)|}{\lambda} \right) v(x)\,dx     \right)\\
        & \leq C \lambda v(Q_j)+ \lambda \int_{Q_j} \Phi_{1} \left( \frac{|f(x)|}{\lambda} \right) v(x)\,dx \\
        & \leq C \int_{Q_j} fv + \lambda \int_{Q_j} \Phi_{1} \left(\frac{|f(x)|}{\lambda} \right) v(x)\,dx .
    \end{split}
\end{align*}

Back to the estimation of $I_{3,2}^{1}$, we arrive at 
\begin{align*}
    \begin{split}
        I_{3,2}^{1} & \leq \frac{C}{\lambda}  \sum_{j}\inf_{Q_j}M_{\Phi_{\varepsilon}, v^{1-q}}u^{*} \int_{Q_j} |b(x)-b_{Q_j}|  v (x) \, dx \\
        &\leq \frac{C}{\lambda}  \sum_{j}\inf_{Q_j}M_{\Phi_{\varepsilon}, v^{1-q}}u^{*} \left(    2   \int_{Q_j}f(x)v(x)\,dx   + \lambda \int_{Q_j} \Phi_{1} \left(\frac{|f(x)|}{\lambda} \right) v(x)\,dx     \right) \\ 
        & \leq C \sum_{j} \frac{1}{\lambda}\int_{Q_j} f(x)M_{\Phi_{\varepsilon}, v^{1-q}}u^{*}(x)v(x)  \,dx + C\sum_{j} \int_{Q_{j} }\Phi_{1} \left(\frac{|f(x)|}{\lambda} \right)M_{\Phi_{\varepsilon}, v^{1-q}}u^{*} (x) v(x)  \,dx.
    \end{split}
\end{align*}
Since $t\leq  \Phi_{1}(t)$ and $M_{\Phi_\varepsilon , v^{1-q}} u^{*} \leq C M_{\Phi_{1+\varepsilon}, v^{1-q}} u \leq w(x)$, it follows that 
\begin{align*}
   I_{3,2}^{1} \leq C\int_{\mathbb{R}^n} \Phi_{1} \left(\frac{|f(x)|}{\lambda} \right)   w (x) v(x) \,dx,
\end{align*}
which completes the proof for the case $m=1$. 

At last, we prove the superior order case by an induction argument. Let $m>1$ and suppose  $T_{b}^{k}$ satisfies the thesis for all $1 \leq k \leq m-1$. Considering the same Calderón-Zygmund decomposition as before, we obtain 
\begin{align*}
    \begin{split}
            uv\left( \left\{ x \in \mathbb{R}^n : \frac{|T^{m}_{b}(fv)(x)|}{v(x)} > \lambda \right\} \right) & \leq uv\left( \left\{ x \in \mathbb{R}^n \backslash \Omega^*: \frac{|T^{m}_{b}(gv)(x)|}{v(x)} > \frac{\lambda}{2} \right\} \right) +uv(\Omega^{*}) \\
    &\hspace{5mm} + uv\left( \left\{ x \in \mathbb{R}^n \backslash{\Omega^{*}}: \frac{|T^{m}_{b}(hv)(x)|}{v(x)} > \frac{\lambda}{2} \right\} \right)\\
    & = I^{m}_{1}+I^{m}_{2}+I^{m}_{3}.
    \end{split}
\end{align*}

To estimate $I^{m}_{1}$ we proceed as in the estimations of $I_1$ and $I_{1}^{1}$ to obtain 
\begin{align*}
    I_{1}^{m} \leq \frac{C}{\lambda^p} \int_{\mathbb{R}^n} |T^{m}_{b}(gv)(x)|^{p} u^{*}(x) v(x)^{1-p}\,dx \leq \frac{C}{\lambda}\int_{\mathbb{R}^n} |g(x)| M_{\Phi_{m+\varepsilon}, v^{1-q}}u^{*}(x)v(x) \,dx,
\end{align*}
where $u^{*}=u \chi_{\mathbb{R}^{n} \backslash \Omega^{*}}$ and we have used Theorem \ref{FS fuerte para conmutadores}. Similarly as before, the hypotheses on the pair $(u,w)$ allows us to get that 
\begin{align*}
    I_{1}^{m} \leq C \int_{\mathbb{R}^n} \Phi_{m} \left( \frac{f(x)}{\lambda}\right) w(x) \, v(x)\,dx.
\end{align*}

Also as before, the estimation for $I_{2}^{m}$ is derived from the fact that the hypotheses on the weights imply that $(u,w) \in A_1 (v^{1-q})$.

To finalize the proof, we estimate $I_{3}^{m}$. By the formula 
\begin{equation*}
    T^{m}_{b}(hv)= \sum_{j}(b-b_{Q_j})^{m}T(h_{j}v) - \sum_{j} T((b - b_{Q_j})^{m} h_{j}v ) - \sum_{j} \sum_{i=1}^{m-1} C_{m,i}\, T_{b}^{i} ((b - b_{Q_j})^{m-i} h_{j}v ) 
\end{equation*} it follows that 
\begin{align*}
\begin{split}
        I_{3}^{m} &\leq uv\left( \left\{   x\in \mathbb{R}^{n} \backslash\Omega^{*} : \left|  \frac{\sum_{j}(b-b_{Q_j})^{m}T(h_{j}v)}{v}\right| >\frac{\lambda}{6} \right\}  \right)  \\ & \hspace{7mm}+uv\left( \left\{   x\in \mathbb{R}^{n} \backslash\Omega^{*} : \left|  \frac{\sum_{j} T((b - b_{Q_j})^{m} h_{j}v )}{v}\right|  >\frac{\lambda}{6}  \right\}  \right)   \\&  \hspace{14mm}+uv\left( \left\{   x\in \mathbb{R}^{n} \backslash\Omega^{*} : \left|  \frac{\sum_{j} \sum_{i=1}^{m-1}  \,T_{b}^{i} ((b - b_{Q_j})^{m-i} h_{j}v )}{v}\right| >\frac{\lambda}{6C_{0}}  \right\}  \right)  \\
        & =I^{m}_{3,1} +  I^{m}_{3,2}  + I^{m}_{3,3},\end{split}
\end{align*}
where $C_0=\max_{i=1}^{m-1} C_{m,i} $. We now estimate each of these terms. 

Applying Tchebyshev inequality, Tonelli Theorem and condition \eqref{condicion de suavidad} we obtain
\begin{align*}
    \begin{split}
         I^{m}_{3,1} & \leq \frac{C}{\lambda} \int_{\mathbb{R}^{n}\backslash \Omega^{*}} \left| \sum_{j} (b(x) - b_{Q_j})^{m} T(h_j v)(x)   \right| u^{*}(x) \,dx \\
         & \leq \frac{C}{\lambda} \sum_j \int_{Q_j} |h_j(y)| v(y) \int_{\mathbb{R}^{n}\backslash Q_{j}^{*}} |b(x) - b_{Q_j}|^{m} |K(x-y) - K(x-x_{Q_j})| u^{*}_{j}(x) \, dx \,dy\\
&  \leq \frac{C}{\lambda} \sum_j \int_{Q_j} |h_j(y)| v(y) \sum_{k=1}^{\infty}
 \int_{A_{j,k}} |b(x) - b_{Q_j}|^{m} \frac{|y-x_{Q_j} |}{|x-x_{Q_j} |^{n+1}} u^{*}_{j}(x) \, dx\, dy,
    \end{split}
\end{align*} where the sets $A_{j,k}$ and $u_{j}^{*}$ are the ones defined before.

Let $k_0$ the only integer such that $2^{k_0 -1} \leq \sqrt{n} <2^{k_0 }$, we get 
{\small \begin{align*}
    \begin{split}
        \sum_{k=1}^{\infty} \int_{A_{j,k}} |b(x) - b_{Q_j}|^{m} \frac{|y-x_j |}{|x-x_j |^{n+1}} u^{*}_{j}(x) \, dx &  \leq C\sum_{k=1}^{\infty}  \frac{2^{-k}}{(\sqrt{n}\, l_j \, 2^{k})^n} \int_{B(x_j , 2^{k+1}\sqrt{n}\, l_j \, )} |b(x) - b_{Q_j}|^{m}  u^{*}_{j}(x) \, dx  \\
        & \leq C \sum_{k=1}^{\infty} \frac{2^{-k}}{|2^{k+k_{0}+2}Q_j|}\int_{2^{k+k_{0}+2}Q_j}  |b(x) - b_{Q_j}|^{m}  u^{*}_{j}(x) \, dx.
    \end{split}
\end{align*}}

Applying Proposition \ref{BMO valor abs con promedios en 2k Q}, the generalized Hölder inequality and Propositions \ref{Young: igualdad normas a la r} and \ref{BMO comparacion norma phi y bmo}, we obtain that  
\begin{align*}
    \begin{split}
     \frac{1}{|2^{k+k_{0}+2}Q_j|}\int_{2^{k+k_{0}+2}Q_j}  |b(x) - b_{Q_j}|^{m}&  u^{*}_{j}(x) \, dx \\&\leq \frac{C}{|2^{k+k_{0}+2}Q_j|} \int_{2^{k+k_{0}+2}Q_j}  |b(x) - b_{2^{k+k_{0}+2} Q_j}|^{m}  u^{*}_{j}(x) \, dx \\ & \hspace{5mm}+ C (k+k_{0}+2)^{m} Mu^{*}_{j}(y) \\
     & \leq C \| |b-b_{2^{k+k_{0}+2} Q_j}|^{m}\|_{\tilde{\Phi}_{m}, 2^{k+k_{0}+2}Q_j} \| u^{*}_{j}\|_{\Phi_{m},2^{k+k_{0}+2}Q_j } \\
     & \hspace{5mm}+ C (k+k_{0}+2)^{m} Mu^{*}_{j}(y) \\
     & \leq C\| b-b_{2^{k+k_{0}+2} Q_j}\|^{m}_{\tilde{\Phi}_{1}, 2^{k+k_{0}+2}Q_j} \| u^{*}_{j}\|_{\Phi_{m},2^{k+k_{0}+2}Q_j } \\
     & \hspace{5mm}+ C (k+k_{0}+2)^{m} Mu^{*}_{j}(y) \\
     & \leq C  M_{\Phi_{m}} u^{*}_{j}(y).
    \end{split}
\end{align*}

By the fact that $M_{\Phi_{m}} u^{*}_{j} \leq M_{\Phi_{m+\varepsilon}} u^{*}_{j} $, it follows
\begin{align*}
    I^{m}_{3,1} \leq \frac{C}{\lambda} \sum_{j} \int_{Q_j} |h_{j}(y)|M_{\Phi_{m+\varepsilon}} u^{*}_{j}(y) \, v(y) \, dy
\end{align*} and proceeding in a similar way as done with $I_{3,1}^{1}$, we get the desired estimate. 

Applying the case proved for $m=0$ with the pair of weights $(u^{*}, M_{\Phi_{m+\varepsilon},v^{1-q}}u^{*})$, we can estimate  $I_{3,2}^{m}$ by 
\begin{align*}
\begin{split}
     I_{3,2}^{m} &\leq  u^{*}v\left( \left\{ x \in \mathbb{R}^n : \left| \frac{ T \left(\sum_{j} (b-b_{Q_j})^{m}h_{j}v\right)(x)}{v(x)} \right| > \frac{\lambda}{4}   \right\} \right) \\
     & \leq \frac{C}{\lambda} \int_{\mathbb{R}^n} \left| \sum_{j} (b(x)-b_{Q_j})^{m} h_{j}(x) \right| M_{\Phi_{\varepsilon}, v^{1-q}}u^{*}(x) v (x) \, dx\\
& \leq  \frac{C}{\lambda} \sum_{j} \left(\inf_{Q_j}M_{\Phi_{\varepsilon}, v^{1-q}}u^{*} \right)\, \int_{Q_j} |b(x)-b_{Q_j}|^{m} f(x)  v (x) \, dx\\
& \hspace{5mm}+  \frac{C}{\lambda} \sum_{j} \left(\inf_{Q_j}M_{\Phi_{\varepsilon}, v^{1-q}}u^{*}\right) \, f_{Q}^{v} \,\int_{Q_j} |b(x)-b_{Q_j}|^{m} v (x) \, dx.
\end{split}
\end{align*}
Since $v \in RH_{\infty}$, by Lemma \ref{obs norma bmo y bmo p} and $\| \cdot\|_{\text{BMO},m}$ we get that
\begin{align*}
    \begin{split}
        f_{Q_{j}}^{v} \int_{Q_j} |b(x) - b_{Q_j}|^{m} v(x) \,dx &\leq C \frac{v(Q_j)}{|Q_j|} \frac{1}{v(Q_j)} \int_{Q_j} f(x)v(x)\,dx 
 \, \int_{Q_j} |b(x)-b_{Q_j}|^{m} \,dx \\
        & \leq C \int_{Q_j}f(x)v(x) \,dx .
    \end{split}
\end{align*}
For the other summand, by the generalized Hölder inequality and Proposition \ref{Young: igualdad normas a la r} and \ref{BMO comparacion norma phi y bmo}, it follows that 
\begin{align*}
    \begin{split}
        \int_{Q_j} |b(x) - b_{Q_j}|^{m} f(x) v(x) \, dx & \leq v(Q_j)  \|(b-b_{Q_j})^{m} \|_{\tilde{\Phi}_{m}, Q_{j},v} \|f \|_{\Phi_{m}, Q_j,v} \\
        & \leq C v(Q_j) \|f \|_{\Phi_{m}, Q_j,v} \\ & \leq  v(Q_j) \left( \lambda +\frac{\lambda}{v(Q_j)} \int_{Q_j} \Phi_{m} \left(\frac{|f(x)|}{\lambda} \right) v(x) \,dx    \right)\\
        & \leq C\lambda v(Q_j)+ \lambda \int_{Q_j} \Phi_{m} \left( \frac{|f(x)|}{\lambda} \right) v(x) \,dx \\
        & \leq C \lambda \int_{Q_j} \frac{f(x)}{\lambda}v(x)\,dx + \lambda \int_{Q_j} \Phi_{m} \left(\frac{|f(x)|}{\lambda} \right) v(x)\,dx \\
        &\leq C \lambda \int_{Q_j} \Phi_{m} \left(\frac{|f(x)|}{\lambda} \right) v(x)\,dx.
    \end{split}
\end{align*}
Therefore, 
\begin{equation*} 
    \begin{split}
        I_{3,2}^{m} & \leq \frac{C}{\lambda}  \sum_{j}\left( \inf_{Q_j}M_{\Phi_{\varepsilon}, v^{1-q}}u^{*}\right) \int_{Q_j} |b(x)-b_{Q_j}|^{m}  v (x) \, dx \\
        &\leq C  \sum_{j}\left(\inf_{Q_j}M_{\Phi_{\varepsilon}, v^{1-q}}u^{*} \right)  \int_{Q_j} \Phi_{m} \left(\frac{|f(x)|}{\lambda} \right) v (x) \,dx     \\ 
        & \leq C  \sum_{j} \int_{Q_{j} }\Phi_{m} \left(\frac{|f(x)|}{\lambda} \right)   M_{\Phi_{\varepsilon}, v^{1-q}}u^{*}(x) v(x) \,dx \\
        & \leq C \int_{\mathbb{R}^n} \Phi_{m} \left(\frac{|f(x)|}{\lambda} \right)   w(x) v(x) \,dx.
    \end{split}
\end{equation*}

At last, to estimate $I_{3,3}^{m}$ we apply the inductive hypothesis and arrive to 
\begin{align*}
    \begin{split}
        I_{3,3}^{m} & \leq \sum_{i=1}^{m-1}  u^{*} v \left( \left\{ \frac{|T^{i}_{b}(\sum_j (b-b_{Q_j})^{m-i} h_{j} v )(x)|}{v(x)} > \frac{\lambda}{6C_0}   \right\}  \right) \\
        & \leq \sum_{i=1}^{m-1} \int_{\mathbb{R}^n} \Phi_{i} \left(\frac{\sum_j |b(x)-b_{Q_j}|^{m-i} |h_{j}(x)|}{\lambda}\right) M_{\Phi_{i+\varepsilon}, v^{1-q'}}u^{*}(x) v(x) \,dx
    \end{split}
\end{align*}

\begin{align*}
    \begin{split}
     \textcolor{white}{ I_{3,3}^{m} }&   \leq \sum_{i=1}^{m-1} \sum_{j} \int_{Q_j} \Phi_{i} \left(\frac{ |b(x)-b_{Q_j}|^{m-i} |h_{j}(x)|}{\lambda}\right) M_{\Phi_{i+\varepsilon}, v^{1-q}}u_{j}^{*}(x) v(x) \,dx  \\
        &\leq \sum_{i=1}^{m-1} \sum_{j}  \left( \inf_{Q_j} M_{\Phi_{i+\varepsilon}, v^{1-q}} u^{*}_{j} \right) \int_{Q_j} \Phi_{i} \left(\frac{ |b(x)-b_{Q_j}|^{m-i} |h_{j}(x)|}{\lambda} \right)v(x) \,dx.
    \end{split}
\end{align*}
Let $\Psi_{i} (t) \approx (e^{ t^{1/i} }-e)\chi_{(1,\infty)}(t) $, we get that  $\Phi_{m}$, $\Psi_{m-i}$ and $\Phi_{i}$ verify condition \eqref{condicion triplete holder},  
\begin{align*}
     \Phi^{-1}_{m}(t) \Psi^{-1}_{m-i}(t) \leq C \Phi_{i}^{-1}(t).
\end{align*} Therefore, if $C_1$ is the constant from Proposition \ref{BMO comparacion norma phi y bmo}, by \eqref{des Young generalizada} we get that
\begin{align*}\begin{split}
    \int_{Q_j}  \Phi_{i} \left(\frac{  |b(x)-b_{Q_j}|^{m-i} |h_{j}(x)|}{ \lambda} \right)v(x) \,dx & \leq \int_{Q_j} \Phi_{m}\left( \frac{C_1 |h_j (x)|}{ \lambda} \right) v(x)\,dx \\
    & \hspace{7mm} +
    \int_{Q_j} \Psi_{m-i} \left( \frac{|b(x) - b_{Q_j}|^{m-i}}{C_1}\right) v(x) \,dx.
    \end{split}
\end{align*}
To estimate the first integral, since $\Phi_{m}$  is convex and submultiplicative and $v \in RH_{\infty}$, we obtain that 
\begin{align*}\begin{split}
    \int_{Q_j} \Phi_{m} \left( \frac{C_1 |h_j (x)|}{\lambda} \right) v(x)\,dx & \leq C \left(\int_{Q_j} \Phi_{m}\left(\frac{|f(x)|}{\lambda} \right) v(x) \,dx + \int_{Q_j} \Phi_{m} \left( \frac{f_{Q_j}^{v}}{\lambda} \right)v(x) \,dx\right) \\
    & \leq C \left(\int_{Q_j} \Phi_{m}\left(\frac{|f(x)|}{\lambda} \right) v(x) \,dx  + \Phi_{m} \left( \frac{f_{Q_j}^{v}}{\lambda} \right)  v(Q_j) \right)\\ & 
    \leq C \int_{Q_j} \Phi_{m}\left(\frac{|f(x)|}{\lambda} \right) v(x) \,dx.
    \end{split}
\end{align*}
For the second integral, Proposition \ref{BMO comparacion norma phi y bmo} allows us to obtain
\begin{align*}
    \begin{split}
        \int_{Q_j} \Psi_{m-i} \left( \frac{|b(x) - b_{Q_j}|^{m-i}}{C_1}\right) v(x) \,dx & = \int_{Q_j} \Psi \left(\frac{|b(x) - b_{Q_j}|^{m-i}}{C_1} \right) v(x) \,dx \\
        & \leq  \int_{Q_j} \Psi \left(\frac{|b(x)-b_{Q_j}|}{\| b - b_{Q_j} \|_{\Psi, Q_{j},v}}\right) v(x) \,dx \\
        & \leq v(Q_j) \leq C\int_{Q_j} \frac{f(x)}{\lambda}  v(x) \,dx .
    \end{split}
\end{align*}
Bringing together these two estimations, we get that
\begin{align*}
    \begin{split}
        I_{3,3}^{m} & \leq C\sum_{i=1}^{m-1} \sum_{j} \left( \inf_{Q_j} M_{\Phi_{i+\varepsilon}, v^{1-q}}u^{*}_{j} \right) \int_{Q_{j}}\Phi_{m}\left( \frac{|f(x)|}{\lambda}\right) v(x) \,dx \\ & \hspace{7mm} +C \sum_{i=1}^{m-1} \sum_{j}\left( \inf_{Q_j} M_{\Phi_{i+\varepsilon}, v^{1-q}}u^{*}_{j} \right)   \int_{Q_j} \frac{f(x)}{\lambda}v(x) \,dx 
    \end{split}
\end{align*} 
\begin{align*}
    \begin{split}
        \textcolor{white}{ I_{3,3}^{m}} & \leq C\sum_{i=1}^{m-1} \sum_{j}\int_{Q_{j}}\Phi_{m}\left( \frac{|f(x)|}{\lambda}\right) M_{\Phi_{i+\varepsilon}, v^{1-q}}u^{*}_{j}(x) v(x) \,dx\\& \hspace{7mm} + C \sum_{i=1}^{m-1} \sum_{j}\int_{Q_j}\Phi_{m}\left( \frac{|f(x)|}{\lambda}\right) M_{\Phi_{i+\varepsilon}, v^{1-q}}u^{*}_{j}(x)v(x) \,dx \\
        & \leq C \sum_{i=1}^{m-1} \int_{\mathbb{R}^n}\Phi_{m}\left( \frac{|f(x)|}{\lambda}\right) M_{\Phi_{i+\varepsilon}, v^{1-q}}u(x) v(x) \,dx.
    \end{split}
\end{align*}
Since $\Phi_{i+\varepsilon} \prec_{\infty} \Phi_{m+\varepsilon}$ for $1 \leq i \leq m-1$, by the hypotheses on the pair of weights it follows that  $M_{\Phi_{i+\varepsilon}, v^{1-q}}u(x) \leq C M_{\Phi_{m+\varepsilon}, v^{1-q}}u(x) \leq Cw(x)$, and
\begin{equation*}
    I_{3,3}^{m} \leq \int_{\mathbb{R}^n}\Phi_{m}\left( \frac{|f(x)|}{\lambda}\right) w(x) v(x) \,dx,
\end{equation*} which finalizes the proof.
\end{proof}
\section*{Declarations}

\subsection*{Ethical approval}
Not applicable.
\subsection*{Competing interests}
No potential conflict of interest was reported by the authors.
\subsection*{Author's contributions}
All authors whose names appear on the submission made substantial contributions to the conception and design of the work, drafted the work and revised it critically, approved this version to be published, and agree to be accountable for all aspects of the work in ensuring that questions related to the accuracy or integrity of any part of the work are appropriately investigated and resolved.
\subsection*{Funding}
The authors were supported by CAI+D 85320240100008LI (UNL) and PEICID 2023 Nº 181 (Gobierno de la Provincia de Santa Fe).
\subsection*{Availability of data and materials}
Not applicable.

\providecommand{\bysame}{\leavevmode\hbox to3em{\hrulefill}\thinspace}
\providecommand{\MR}{\relax\ifhmode\unskip\space\fi MR }
\providecommand{\MRhref}[2]{%
  \href{http://www.ams.org/mathscinet-getitem?mr=#1}{#2}
}

\end{document}